\documentclass[11pt, letterpaper,final, reqno]{amsart}
\usepackage{amsfonts,amsmath, amsthm, amssymb, latexsym, epsfig,  bbm, mathtools}
\usepackage[english]{babel}
\usepackage[utf8]{inputenc}
\usepackage[all]{xy}
\usepackage{xspace}
\usepackage{comment}
\usepackage{setspace}
\usepackage{enumerate}
\usepackage{stmaryrd}
\usepackage{xcolor}
\usepackage{hyperref}
\usepackage[mathscr]{eucal}
\usepackage[notcite,notref]{showkeys}

	\newcommand{\ZZ}{\mathbb{Z}}
	
	\newcommand{\QQ}{\mathbb{Q}}
	
	\newcommand{\NN}{\mathbb{N}}
	\newcommand{\TT}{\mathbb{T}}

	\newcommand{\Kk}{\mathcal{K}}

    \newcommand{\Mm}{\mathcal{M}}
		
	\newtheorem{thm}{Theorem}[section]
	\newtheorem{cor}[thm]{Corollary}
	\newtheorem{lem}[thm]{Lemma}

	\theoremstyle{definition}
	\newtheorem{dfn}[thm]{Definition}

	\theoremstyle{remark}
	\newtheorem{rmk}[thm]{Remark}

    \numberwithin{equation}{section}

    \newcommand{\lsp}{\operatorname{span}}
    \newcommand{\clsp}{\operatorname{\overline{span}}}
    \newcommand{\supp}{\operatorname{supp}}

\definecolor{aspurple}{rgb}{0.4,0.05,0.35}

\newcommand{\topo}{topological}
\newcommand{\smtx}[1]{\left(\begin{smallmatrix} #1
\end{smallmatrix}\right)}
\newcommand{\wilde}{\widetilde}
\newcommand{\<}{\langle}
\renewcommand{\>}{\rangle}
\newcommand{\pre}[1]{{}_{#1}}
\newcommand{\ibm}{imprimitivity bimodule}

\newcommand{\midtext}[1]{\quad\text{#1}\quad}

\newcommand{\deltag}{\delta_\Gamma}
\newcommand{\id}{\text{id}}
\newcommand{\cstg}{\ensuremath{C^*(\Gamma)}}

\newcommand{\csta}{\ensuremath{C^*}-algebra}
\newcommand{\inv}{^{-1}}
\newcommand{\xt}{\otimes}

\newcommand{\CE}{conditional expectation}
\newcommand{\rep}{representation}

\newcommand{\sat}{saturated }

\begin{document}
	\title{Morita Equivalence of Graded $C^*$-algebras}

    \author[J.~Quigg]{John Quigg}
        \address{School of Mathematical and Statistical Sciences \\ Arizona State University \\ Tempe, Arizona 85287 USA}
        \email{quigg@asu.edu}

    \author[E.~Ruiz]{Efren Ruiz}
        \address{Department of Mathematics\\University of Hawaii,
Hilo\\200 W. Kawili St.\\
Hilo, Hawaii\\
96720-4091 USA}
        \email{ruize@hawaii.edu}

        \author[A.~Sims]{Aidan Sims}

\address[A. Sims]{School of Mathematics and Statistics\\
University of New South Wales\\
Sydney NSW  2052\\
Australia}
\email[A. Sims]{aidan.sims@unsw.edu.au}

               \date{\today}

\thanks{This work was initiated while Ruiz and Sims were attending the workshop \emph{Cartan Subalgebras in Operator Algebras, and Topological Full Groups} (\# 24w5175) and completed while Ruiz and Sims attended the workshop in MATRIX \emph{Combinatorial *-Algebras}.  This work has benefited from research visits of Ruiz (Arizona State University and University of New South Wales) and Sims (University of Hawai'i at Hilo).  We thank all institutions for their hospitality.  In particular, Ruiz thanks Jack Spielberg for hosting his visit at Arizona State University.  This work was supported by ARC Discovery Project DP220101631 and by an NSF grant DMS 2452325.}
\begin{abstract}
We define two notions of Morita equivalence for graded $C^*$-algebras (graded Morita equivalence and homogeneous Morita equivalence) and provide Brown--Green--Rieffel Stabilization Type Theorems for both notions of graded equivalence.  We apply our results to finite regular graphs by establishing an explicit connection between graded $C^*$-algebras and coactions.  Lastly, we incorporate Cartan subalgebras with totally disconnected spectra and obtain Brown--Green--Rieffel Stabilization Type Theorems for these cases.
\end{abstract}

\maketitle

\section{Introduction}

We study Morita equivalence of \csta s $A$ that are graded over a discrete group $\Gamma$.
We insist throughout that the grading be topological in Exel's sense,
namely there should exist a compatible conditional expectation
onto the unit fibre $A_e$,
and we sometimes want gradings that are \emph{strong}
in the sense of having a compatible conditional expectation
and the (closed spans of) products of
fibres $A_g$ and $A_h$ give all of $A_{gh}$.

Given two graded \csta s, we consider appropriately graded \ibm s,
and give conditions allowing us to prove two ``graded versions'' of the Brown--Green--Rieffel Stabilization theorem. Write $\Kk$ for the algebra of compact operators on the separable Hilbert space with orthonormal basis indexed by the group over which $A$ is graded. In our proofs, we consider two gradings on the stabilization $A\xt \Kk$ of $A$: the ``univalent'' grading
with subspaces $(A\xt \Kk)_g=A_g\xt \Kk$; and the ``bivalent'' grading with subspaces
and $(A\xt \Kk)_h=\clsp_{g,k}\{A_{ghk\inv}\xt \mathbf{e}_{g,k}\}$, respectively,
in our proofs.
If the gradings come from coactions, these results give ``dual'' (i.e., for coactions rather than actions) versions of Combes' equivariant Brown--Green--Rieffel theorem.

Gradings and coactions are also intimately tied to Fell bundles.
In particular, every graded \csta\ is generated by a faithful \rep\ of an associated Fell bundle.
Although the relationship between topological gradings and coactions is not perfect in general, there is no problem for amenable $\Gamma$,
and we exploit this as follows:
given coactions $(A_i,\delta_i)$ for $i=1,2$ and a graded $A_1$--$A_2$ Morita equivalence,
we show that the coactions are Morita equivalent\footnote{and vice-versa}
by using the associated
Fell bundle on the linking algebra to produce a coaction,
then extracting a $\delta_1$--$\delta_2$ Morita equivalence from its upper right-hand corner.

We make the connection with coactions explicit so that we can apply our results to finite regular graphs, using gauge actions.
We finish by showing how to incorporate Cartan subalgebras with totally disconnected spectra.

\section{Graded C*-algebras}

\begin{dfn}
Let $A$ be a $C^*$-algebra and let $\Gamma$ be a discrete group.  A \emph{$\Gamma$-grading} of $A$ is a collection $(A_g)_{g \in \Gamma}$ of closed linear subspaces of $A$ satisfying the following:
\begin{enumerate}
\item $A_g A_h \subseteq A_{gh}$;

\item $A_g^*= A_{g^{-1}}$;  and

\item $A$ is densely spanned by $\bigcup_{g \in \Gamma} A_g$.
\end{enumerate}

A $C^*$-algebra with a $\Gamma$-grading is called a \emph{$\Gamma$-graded $C^*$-algebra}, or just a \emph{graded $C^*$-algebra}.
\end{dfn}

\begin{dfn}
Let $A$ be a $C^*$-algebra and let $\Gamma$ be a discrete group.   A $\Gamma$-grading of $A$, $(A_g)_{g \in \Gamma}$,
\begin{enumerate}
\item is \emph{topological} if there exists a conditional expectation from $A$ to $A_{e}$ vanishing on all $A_g$ for all $g \neq e$ and

\item is \emph{strong} if $A_g A_h$ is dense in $A_{gh}$ for all $g, h \in \Gamma$.
\end{enumerate}

A $C^*$-algebra with a $\Gamma$-grading that is topological is called a \emph{topologically graded $C^*$-algebra}.  A $C^*$-algebra with a $\Gamma$-grading that is strong and topological is called a \emph{\sat graded $C^*$-algebra}. %\jq{I recommend using a different term for gradings that are both strong and topological; using ``strong'' again introduces a confusing ambiguity.}
\end{dfn}

\begin{lem}\label{lem-strgr-approx-id}
If $A$ is a $\Gamma$-graded $C^*$-algebra, then every approximate identity for $A_e$ is an approximate identity for $A$.
\end{lem}

\begin{proof}
Fix an approximate identity $(e_\lambda)_{\lambda \in \Lambda}$ for $A_e$. Since $A$ is graded, it suffices to show that if $g \in \Gamma$ and $b \in A_g$, then $e_\lambda b \to b$. For this, note that $bb^* \in A_e$, and so
\begin{align*}
\|e_\lambda b-b\|^2
&=\|(e_\lambda b-b)(e_\lambda b-b)^*\|
\\&=\|e_\lambda bb^*e_\lambda-e_\lambda bb^*-bb^*e_\lambda+bb^*\|
%\\&
\xrightarrow{\lambda}0.\qedhere
\end{align*}
\end{proof}

We introduce a relation between graded \csta s that is quite similar to what Muhly--Williams would call an \emph{equivalence} between the associated Fell bundles (see \cite[Section~6]{mw}).

\begin{dfn}[Graded equivalence]\label{dfn:graded equivalence}
Let $A$ and $B$ be topologically $\Gamma$-graded $C^*$-algebras.  An $A$--$B$-bimodule $X$ is \emph{graded} if $X$ is a left Hilbert $A$-module and a right Hilbert $B$-module, and there are closed linear subspaces $(X_g)_{g \in \Gamma}$ of $X$ such that
\begin{enumerate}
\item\label{it:graded actions} $A_{g} X_h B_{g'} \subseteq X_{ghg'}$ for all $g, h, g' \in \Gamma$;
\item for all $g,h\in \Gamma$,
\[
    _A\langle X_g,X_h\rangle \subset A_{gh\inv}\midtext{and}\langle X_g,X_h\rangle_B\subset B_{g\inv h};
\]
and
\item $X$ is densely spanned by $\bigcup_{g \in \Gamma} X_g$.
\end{enumerate}
We say that $A$ and $B$ are \emph{graded Morita equivalent} if there exists an $A$--$B$-imprimitivity bimodule $X$ that is graded.
\end{dfn}

\begin{rmk}\label{rmk-lin-indep}
If $X$ is graded as above, then the homogeneous subspaces $X_g$ are linearly independent. To see this, $E \colon B \to B_e$ is the conditional expectation of the topologically graded $C^*$-algebra $B$, and suppose that $\sum_g x_g = 0$ with each $x_g \in X_g$. Then for any $g \in G$ we have $0 = E(\langle x_g, \sum_{h \in G} x_h\rangle_B) = \langle x_g, x_g\rangle_B$.
\end{rmk}

We now show that there is an induced topological grading on the linking algebra (see \cite[page~50]{RW-Morita-Eq}) of a graded imprimitivity bimodule.

\begin{lem}\label{lem-linkalg-grd}
Let $A_1$ and $A_2$ be topologically $\Gamma$-graded $C^*$-algebras and let $X$ be a graded $A_1$--$A_2$-imprimitivity bimodule.  Then there exists a contractive linear map $E_0 \colon X \to X_{e}$ that pointwise fixes $X_e$ and annihilates $X_g$ for $g \not= e$. The map $E \colon L(X) \to L(X)_e$ defined by
\[
E\left( \begin{pmatrix} a & x \\ \widetilde{y} & b \end{pmatrix}\right)
    = \begin{pmatrix} E_{A_1}(a) & E_0(x) \\ E_0(y)\widetilde{\;} & E_{A_2}(b) \end{pmatrix}
\]
is a conditional expectation on the linking algebra $L(X)$.  Consequently, $L(X)$ is a topologically $\Gamma$-graded $C^*$-algebra with graded subspaces given by
\[
L(X)_g = \Big( \begin{matrix} A_g & X_g \\ {(X_{g^{-1}})\widetilde{\;}} & B_g\end{matrix}\Big).
\]
\end{lem}

\begin{proof}
A computation shows that $L(X)$ is a $\Gamma$-graded $C^*$-algebra.  Thus, we are left to show the existence of a contractive linear map $E_0 \colon X \to X_{e}$ such that the map $E \colon L(X) \to L(X)_e$ defined in the lemma is a conditional expectation.

Since the $X_g$ are linearly independent, there is a unique linear map $E_0 \colon \operatorname{span} \big(\bigcup_g X_g\big) \to X_e$ such that $E_0|_{X_e} = \operatorname{id}_{X_e}$ and $E_0(X_g) =  \{0\}$ for $g \not= e$.

We show that $E_0$ is norm-decreasing. Fix a finite linear combination $x = \sum_ g x _g$ with each $x_g \in X_g$. Since $\langle x_g, x_g \rangle_{A_2} \geq 0$ for all $g$,
\begin{align}
0 \leq \langle E_0(x), E_0(x) \rangle_{A_2} &= \langle x_e, x_e \rangle_{A_2} \nonumber\\
							&\leq \sum_{ g } \langle x_g, x_g \rangle_{A_2} \nonumber\\
							&= \sum_{g, h } E_{A_2} (\langle x_g, x_h \rangle_{A_2} ) \nonumber\\
							&= E_{A_2} ( \langle x , x \rangle_{A_2} ).\label{eq:estimate}
\end{align}
Thus, $\| \langle E_0(x), E_0(x) \rangle_{A_2}  \| \leq \| E_{A_2} ( \langle x , x \rangle_{A_2} ) \| \leq \|  \langle x , x \rangle_{A_2} \|$ which proves that $E_0$ is norm-decreasing.  Hence $E_0$ extends uniquely to a contractive linear map $E_0 : X = \clsp \bigcup_g X_g \to X_e$ which we again denote by $E_0$. For $x \in X$, continuity and the computation~\eqref{eq:estimate} give $0 \leq \langle E_0(x), E_0(x) \rangle_{A_2} \leq E_{A_2}( \langle x , x \rangle_{A_2})$.

Let $E \colon L(X) \to L(X)_e$ be as in the statement of the lemma. Then $E$ is a linear map that pointwise fixes $L(X)_e$ and annihilates $L(X)_g$ for all $g \neq e$.  We show that $E$ is norm decreasing.  Consider an element $A = \left( \begin{smallmatrix} a & x \\ \widetilde{y} & b \end{smallmatrix}\right)$ of $L(X)$.  Then $E(a)$, regarded as an element of $\mathcal{L}( X_e \oplus (A_2)_e)$, is the map
\[
 \begin{pmatrix} z \\ c \end{pmatrix} \in X_e \otimes (A_2)_e \mapsto \begin{pmatrix} E_{A_1}(a) z + E_0(x) c \\ \langle E_0(y) , z \rangle_{A_2} + E_{A_2}(b)c \end{pmatrix}.
\]
By approximating $a, b, x, y$ by linear combinations of homogeneous elements and using that $E_{A_2}$ is a conditional expectation, we see that for all $z \in X_e$ and for all $c \in (A_2)_e$,
\[
E(A) \begin{pmatrix} z \\ c \end{pmatrix}
    = \begin{pmatrix} E_0 (az + xc) \\ E_{A_2} ( \langle y, z \rangle_{A_2} + bc ) \end{pmatrix}
    = (E_0 \oplus E_{A_2}) \left(\begin{pmatrix} a & x \\ \widetilde{y} & b \end{pmatrix} \begin{pmatrix} z \\ c \end{pmatrix}\right).
\]
So, writing $m_A : X_e \oplus (A_2)_e \to X_e \oplus (A_2)_e$ for left-multiplication by $A$ and $\iota$ for the inclusion $X_e \oplus (A_2)_e \hookrightarrow X \oplus A_2$, we have $E(A) = (E_0 \oplus E_{A_2})  \circ m_A \circ \iota$. Since $\|\iota\| = 1$ and $\|m_A\| = \|A\|$, it therefore suffices to show that $\|E_0 \oplus E_{A_2}\| \le 1$.  Let $z \in X$ and $c \in A_2$.  Since
\begin{align*}
0\leq \left\langle \begin{pmatrix} E_0(z) \\ E_{A_2}(c) \end{pmatrix} , \begin{pmatrix} E_0(z) \\ E_{A_2}(c) \end{pmatrix}\right\rangle_{A_2} &= \langle E_0(z) , E_0(z) \rangle_{A_2} + \langle E_{A_2}(c), E_{A_2}(c) \rangle_{A_2} \\
		&\leq E_{A_2} ( \langle z , z \rangle_{A_2} + \langle c, c \rangle_{A_2} ) \\
		&= E_{A_2} \bigg( \bigg\langle \bigg(\begin{matrix} z \\ c \end{matrix}\bigg) , \bigg(\begin{matrix} z \\ c \end{matrix}\bigg)\bigg\rangle_{A_2}  \bigg),
\end{align*}
we have
\[
\bigg\|  \bigg\langle \begin{pmatrix} E_0(z) \\ E_{A_2}(c) \end{pmatrix} , \begin{pmatrix} E_0(z) \\ E_{A_2}(c) \end{pmatrix} \bigg\rangle_{A_2}\bigg\| \leq \bigg\| \bigg\langle \bigg(\begin{matrix} z \\ c \end{matrix}\bigg) , \bigg(\begin{matrix} z \\ c \end{matrix}\bigg)\bigg\rangle_{A_2}  \bigg\|
\]
as required, completing the proof that $E$ is norm decreasing.

Since $E$ is a contractive projection from $L(X)$ onto $L(X)_e$, by \cite{Tomi1957}, $E$ is a conditional expectation.  Thus, $L(X)$ is a topologically graded $C^*$-algebra.
\end{proof}

We now specialise to graded imprimitivity bimodules for which the zero component is an imprimitivity bimodule.

\begin{dfn}[Homogeneous Graded Morita Equivalence]
Let $A$ and $B$ be topologically $\Gamma$-graded $C^*$-algebras.  We say $A$ and $B$ are \emph{homogeneously graded Morita equivalent} if there exists an $A$--$B$-imprimitivity bimodule $X$ that is graded, such that $X_e$ is an $A_e$--$B_e$-imprimitivity bimodule.
\end{dfn}

\begin{dfn}\label{dfn-univalent-grading}
Let $A$ be topologically graded $C^*$-algebra with grading $(A_g)_{g \in \Gamma}$. The closed subspaces $(A \otimes \Kk)_g := A_g \otimes \Kk$ of $A \otimes \Kk$ constitute a topological grading on $A \otimes \Kk$ which we call the \emph{univalent grading}.
\end{dfn}

\begin{rmk}
A $\Gamma$-grading of a $C^*$-algebra $A$ is frequently determined by a coaction of $\Gamma$ on $A$ (see the discussion following Corollary~\ref{cor-BGR-gradedME}): the coaction is uniquely determined by $\delta|_{A_g} = (a \mapsto a \otimes u_g)$.
In this picture, the univalent grading of $A \otimes \Kk$ corresponds to the coaction $\delta \otimes \operatorname{id}_{\Kk}$.
\end{rmk}

We are now ready to prove our first Brown-Green-Rieffel Stabilisation Theorem for topologically graded $C^*$-algebras.

\begin{thm}\label{thm-BGR}
Let $A$ and $B$ be $\sigma$-unital $C^*$-algebras that are topologically graded.  Then $A$ and $B$ are homogeneously graded Morita equivalent if and only if $A \otimes \Kk$ and $B \otimes \Kk$, endowed with the univalent gradings, are graded isomorphic.
\end{thm}
\begin{proof}
First suppose that $P$ is a full projection in the multiplier algebra $\Mm(B_e)$ of the trivially graded component of $B$ and that $A = P B P$. Then $A_e = P B_e P$. By \cite[Lemma~2.5]{Brown}, there is a partial isometry $V$ in $\Mm(B_e \otimes \Kk)$ such that $V^*V = 1_{\Mm(B_e \otimes \Kk)}$ and $VV^* = P \otimes 1$. Since $B_e$ contains an approximate identity for $B$, the inclusion $B_e \hookrightarrow B$ extends to an inclusion $\Mm(B_e) \hookrightarrow \Mm(B)$, so we can regard $V$ as an element of $\Mm(B)$. Hence $\operatorname{Ad}_V : x \mapsto V x V^*$ is an isomorphism $B \otimes \Kk \to P B P \otimes \Kk = A \otimes \Kk$. Write $V = \lim_n v_n$ as a strict limit of elements $v_n$ of $B_e \otimes \Kk$. If $b \in B_g$ for some $g \in \Gamma$ and $T \in \Kk$, then $b \otimes T \in (B \otimes \Kk)_g$ and so $v_n (b \otimes T) v_n^* \in (B \otimes \Kk)_g$. Hence $V (b \otimes T) V^* = \lim_n v_n (b \otimes T)v_n^* \in (A \otimes \Kk)_g = A_g \otimes \Kk$. Hence $\operatorname{Ad}_V$ is a graded isomorphism.

Now suppose that $A$ and $B$ are just homogeneously graded Morita equivalent; let $X$ be the graded imprimitivity bimodule implementing this homogeneous graded Morita equivalence. Consider the linking algebra
\[
L = L(X) = \Big( \begin{matrix} A & X \\ \widetilde{X} & B\end{matrix}\Big).
\]
Then, by Lemma~\ref{lem-linkalg-grd}, $L$ is $\Gamma$-graded with graded subspaces
\[
L_g = \Big( \begin{matrix} A_g & X_g \\ {(X_{g^{-1}})\widetilde{\;}} & B_g\end{matrix}\Big).
\]
Since $X_e$ is a Morita equivalence between $A_e$ and $B_e$, we see that $A_e$ and $B_e$ are both full corners of $L_e$, so the projection $P_A = \big(\begin{smallmatrix} 1_{\mathcal{M}(A)} & 0 \\ 0 & 0\end{smallmatrix}\big)$ is a full projection in $\Mm(L_e)$. So the first paragraph yields a graded isomorphism $A \otimes \Kk \cong L \otimes \Kk$. Similarly, using the full projection $P_B = \big(\begin{smallmatrix} 0 & 0 \\ 0 & 1_{\mathcal{M}(B)}\end{smallmatrix}\big)$ in $\Mm(L_e)$, we have a graded isomorphism $B \otimes \Kk \cong L \otimes \Kk$. Combining these isomorphisms yields the desired graded isomorphism $A \otimes \Kk \cong B \otimes \Kk$.

Suppose $\phi \colon A \otimes \Kk \to B \otimes \Kk$ is a graded isomorphism with respect to the univalent gradings on $A \otimes \Kk$ and $B \otimes \Kk$. Write $\{ \mathbf{e}_{i,j}\}$ for the standard system of matrix units for $\Kk$. Let $p \coloneqq \varphi(1_{\mathcal{M}(A_1)} \otimes \mathbf{e}_{1,1})$ and $q \coloneqq 1_{\mathcal{M}(A_2)} \otimes \mathbf{e}_{1,1}$, and let $X \coloneqq p(A_2 \otimes \Kk)q$ with the standard inner products and actions.  Then $X$ is a graded $A_1$--$A_2$-imprimitivity bimodule implementing a homogeneously graded Morita equivalence between $A_1$ and $A_2$.
\end{proof}

We now show that for \sat graded $C^*$-algebras, any graded imprimitivity bimodule implements a homogeneously graded Morita equivalence.

\begin{cor}\label{cor-homogeneous for free}
Let $A$ and $B$ be $\sigma$-unital $C^*$-algebras that are \sat graded $C^*$-algebras. Suppose that $X$ is a graded $A$--$B$-imprimitivity bimodule.  Then $X_e$ is an $A_e$--$B_e$-im\-prim\-i\-tiv\-i\-ty bimodule, whereby $A$ and $B$ are homogeneously graded Morita equivalent. Consequently, $A \otimes \Kk$ and $B \otimes \Kk$, endowed with the univalent gradings, are graded isomorphic.
\end{cor}
\begin{proof}
We show that
\[
    A_e = {_{A_e}\langle X_e, X_e \rangle}  \quad \text{and} \quad B_e =  \langle X_e, X_e \rangle_{B_e}.
\]
Since $B$ is strongly graded, $B_{g^{-1}} B_{g} = B_e$ for each $g \in \Gamma$.  That $A = {_{A}\langle X, X \rangle}$ ensures that $A_e =E( {_{A}\langle X, X \rangle})$.  Hence, using Lemma~\ref{lem-strgr-approx-id} at the third equality, we compute
\begin{align*}
A_e &= E( {}_{A} \langle X, X \rangle)%\overset{\phantom{\ref{lem-strgr-approx-id}}}{=}
     = \clsp \{ E( {_{A}\langle x, y \rangle} ) : x , y \in X \} \\%&\overset{\eqref{lem-strgr-approx-id}}{=}
    &= \clsp \bigcup_{g \in \Gamma}  E({_{A}\langle X_g B_e, X_g \rangle}) %\overset{\phantom{\ref{lem-strgr-approx-id}}}{\subseteq}
     \subseteq \clsp \bigcup_{g \in \Gamma } E({_{A}\langle X_g B_{g^{-1}} B_g, X_g \rangle}) \\ %&\overset{\phantom{\eqref{lem-strgr-approx-id}}}{=}
    &= \clsp \bigcup_{ g \in \Gamma} E({_{A}\langle X_g B_{g^{-1}} , X_g B_{g^{-1}} \rangle}  )%\overset{\phantom{\eqref{lem-strgr-approx-id}}}{=}
    = \clsp E({_{A_e}\langle X_e, X_e\rangle} ) %\overset{\phantom{\ref{lem-strgr-approx-id}}}{\subseteq}
    \subseteq A_e.
\end{align*}
Since $A$ is strongly graded, $A_{g} A_{g^{-1}} = A_e$ for each $g \in \Gamma$.  Since $B =  \langle X, X \rangle_B$, we have $B_e = E( \langle X, X \rangle_B )$.  Thus, again using Lemma~\ref{lem-strgr-approx-id} at the third equality, we have
\begin{align*}
B_e &%\overset{\phantom{\eqref{lem-strgr-approx-id}}}{=}
    = E(\langle X, X \rangle_B) %\overset{\phantom{\eqref{lem-strgr-approx-id}}}{=}
    = \clsp \{ E( \langle x, y \rangle_A ) : x , y \in X \} \\ %&\overset{\eqref{lem-strgr-approx-id}}{=}
    &= \clsp \bigcup_{g \in \Gamma}  E(\langle A_e X_g , X_g \rangle_B) %\overset{\phantom{\ref{lem-strgr-approx-id}}}{\subseteq}
    \subseteq \clsp \bigcup_{g \in \Gamma } E(\langle A_g A_{g^{-1}} X_g ,X_g \rangle_B) \\ %&\overset{\phantom{\eqref{lem-strgr-approx-id}}}{=}
    &= \clsp \bigcup_{ g \in \Gamma} E(\langle A_{g^{-1}} X_g  , A_{g^{-1}} X_g \rangle  ) %\overset{\phantom{\eqref{lem-strgr-approx-id}}}{=}
    = \clsp  E(\langle X_e, X_e\rangle_{B_e} ) %\overset{\phantom{\ref{lem-strgr-approx-id}}}{\subseteq}
    \subseteq B_e.
\end{align*}
Thus, $X_e$ is an $A_e$--$B_e$-imprimitivity bimodule and $A$ and $B$ are homogeneously graded Morita equivalent.  The last part of the theorem follows from Theorem~\ref{thm-BGR}.
\end{proof}

Let $A$ be a $C^*$-algebra,  let $X$ be a right Hilbert $A$-module, and let $T$ be an adjointable operator on $X$. Since $\|Ty\| = \sup_{\|x\| = 1} \|\langle x, Ty\rangle_A\|$ (see \cite[page~7]{Lance} or \cite[Equation~(2.14)]{RW-Morita-Eq}), we have
\begin{equation}\label{eq-operator-norm}
\|T\| = \sup\{ \| \langle x , Ty \rangle : \|x\| \leq 1, \| y\| \leq 1\}.
\end{equation}

Let $A$ be a topologically $\Gamma$-graded $C^*$-algebra with conditional expectation $E_A \colon A \to A_e$.  By \cite[Corollary~19.6]{ExelBook}, for each $g \in \Gamma$, there exists a contractive linear map $E_A^g \colon A \to A_g$ such that for any finite sum $x = \sum_{ h } a_h$, $E_A^g(x) = a_g$.  Moreover for all $x \in A$ and for all $b \in A_h$,
\[
    E_A^g(bx) = bE_A^{h^{-1}g}(x) \quad \text{and} \quad E_A^g(xb) = E_{A}^{gh^{-1}}(x)b.
\]

\begin{lem}\label{lem-gradstable}
Let $\Gamma$ be a discrete group and let $A$ be a topologically $\Gamma$-graded $C^*$-algebra.  Then there is a topological $\Gamma$-grading on $A \otimes \Kk(\ell^2(\Gamma))$, which we call the \emph{bivalent grading}, whose spectral subspaces are
\[
(A \otimes \Kk(\ell^2(\Gamma)))_h = \overline{\operatorname{span}} \bigcup_{ g,k} \{ a \otimes \mathbf{e}_{g,k} : a \in A_{ghk^{-1}} \}.
\]
The associated conditional expectation $E_{\operatorname{biv}} \colon A \otimes \Kk(\ell^2(\Gamma)) \to (A \otimes \Kk(\ell^2(\Gamma)))_e$ satisfies
\[
E_{\operatorname{biv}} (a \otimes \mathbf{e}_{g,k} )= E_{A}^{gk^{-1}}(a) \otimes \mathbf{e}_{g,k}.
\]
\end{lem}

\begin{proof}
Let $a \in A_{ghk^{-1}}$ and let $b \in A_{\ell m n^{-1}}$.  Then $ab \in A_{ghmn^{-1}}$ if $k = \ell$ which implies
\[
(a \otimes \mathbf{e}_{g,k})( b \otimes \mathbf{e}_{\ell, n}) = \delta_{ k , \ell} (ab \otimes \mathbf{e}_{g, n}) \in ( A \otimes \Kk(\ell^2(\Gamma)))_{hm}.
\]
Thus, $( A \otimes \Kk(\ell^2(\Gamma)))_{h} ( A \otimes \Kk(\ell^2(\Gamma)))_{m} \subseteq ( A \otimes \Kk(\ell^2(\Gamma)))_{hm}$.  It is clear that $( A \otimes \Kk(\ell^2(\Gamma)))_{h}^* = ( A \otimes \Kk(\ell^2(\Gamma)))_{h^{-1}}$.  If $a \in A_h$, then
\[
a \otimes \mathbf{e}_{g, k} \in ( A \otimes \Kk(\ell^2(\Gamma)))_{g^{-1} hk}.
\]
Consequently, $\overline{\operatorname{span} } \bigcup_{ g } ( A \otimes \Kk(\ell^2(\Gamma)))_{g} = A \otimes \Kk(\ell^2(\Gamma))$.

We are left to show that there is a conditional expectation $E_{\operatorname{biv}} \colon A \otimes \Kk(\ell^2(\Gamma)) \to (A \otimes \Kk(\ell^2(\Gamma)))_e$ satisfying the desired formula. Since the subspaces $\{A \otimes \mathbf{e}_{g,h} : g,h \in \Gamma\}$ are linearly independent, there is a linear map $E^0_{\operatorname{stb}} \colon \lsp\{a \otimes \mathbf{e}_{g,h} : a \in A, g,h \in \Gamma\} \to (A \otimes \Kk(\ell^2(\Gamma)))_e$ such that $E_{\operatorname{biv}} (a \otimes \mathbf{e}_{g,k} )= E_{A}^{gk^{-1}}(a) \otimes \mathbf{e}_{g,k}$ for all $a,g,h$. To show that $E^0_{\operatorname{stb}}$ extends to the desired conditional expectation, we need a little setup.

For $b \in \bigoplus_{g \in \Gamma} A$ and $g \in \Gamma$, we write $b_g$ for the $g$\textsuperscript{th} entry of $b$. The diagonal homomorphism $j_A \colon A \hookrightarrow \mathcal{L}(\bigoplus_{g \in \Gamma} A)$ such that $(j_A(a)b)_g = a b_g$ and the canonical homomorphism $j_{\mathcal{K}} \colon \Kk(\ell^2(\Gamma)) \to \mathcal{L}(\bigoplus_{g \in \Gamma} A)$ such that $(j_{\mathcal{K}}(\mathbf{e}_{g,h}) \cdot b)_k = \delta_{g,k} b_h$ have commuting ranges, and so the universal property of the tensor product yields a homomorphism $j = j_A \otimes j_{\mathcal{K}} \colon A \otimes \Kk(\ell^2(\Gamma)) \to \mathcal{L}(\bigoplus_{g \in \Gamma} A)$ such that
\[
j \bigg( \sum_{g,h \in \Gamma} a(g,h) \otimes \mathbf{e}_{g,h} \bigg) ( z_h) = \bigg( \sum_{ h} a(g,h) z_h \bigg).
\]
Since $\mathcal{K}(\ell^2(\Gamma))$ is simple and nuclear and since $j_A$ is injective, $j$ is injective. If $a = \sum_{g,h \in \Gamma} a(g,h) \otimes \mathbf{e}_{g,h}$ belongs to $(A \otimes \Kk(\ell^2(\Gamma)))_e$, then each $a(g,h) \in A_{gh^{-1}}$, and so each $a(g,h) z_h \in A_{g}$. Hence $X \coloneqq  \bigoplus_{g \in \Gamma} A$ is invariant for $j\big((A \otimes \Kk(\ell^2(\Gamma)))_e\big)$. Thus, $j$ restricts to a homomorphism $\iota \colon   (A \otimes \Kk(\ell^2(\Gamma)))_e \to \mathcal{L}(X)$. We claim that $\iota$ is injective.

To see this, suppose that $x = \sum_{g,h \in \Gamma} a(g,h) \otimes \mathbf{e}_{g,h} \in  A \otimes \Kk(\ell^2(\Gamma))_e$ satisfies $\iota(x) = 0$.  Then for all $(z_k) \in X$ and for all $g \in \Gamma$, we have $\sum_{ h} a(g,h) z_h = 0$.  Since the $A_g$ are linearly independent, we have $a(g,h) z_h = 0$.  Consequently, for all $g, h \in \Gamma$ and for all $z \in A_h$, we have $a(h,g) z = 0$.  Let $g, h \in \Gamma$, and write $[A_hA_{h^{-1}}]$ for the closed linear span of $A_{h}A_{h^{-1}}$, and fix an approximate identity $(y_i)$ of $[A_hA_{h^{-1}}]$. Then $A_{gh^{-1}}$ is a right Hilbert $[A_hA_{h}^{-1}]$-module, and so $a(g,h) y_i \to a(g,h)$.  Since $a(g,h) z =0$ for all $z \in A_h$, we have $a(g,h)y_i = 0$ for all $i$.  Hence, $a(g,h)=0$.  Consequently, $x = 0$. So $\iota$ is injective as claimed.

Now, fix a finite sum $x = \sum_{g,h \in \Gamma} a(g,h) \otimes \mathbf{e}_{g,h} \in A \otimes \Kk (\ell^2(\Gamma))$. Let $y = \sum_{g,h \in \Gamma} E^{gk^{-1}}(a(g,h)) \otimes \mathbf{e}_{g,h}$. Equation~\eqref{eq-operator-norm} gives
\begin{align*}
\|E^0_{\operatorname{stb}}(x)\|
    &= \| \iota\left( y \right)  \|\\
    &= \sup \{  \| \langle (w_k) , \iota (y) (z_k) \rangle\| : \| (w_k)\| \leq 1, \| (z_k)\| \leq 1 \} \\
	&= \sup \bigg\{ \bigg\| \sum_{g, h} w_g^* E_A^{gh^{-1}} ( a(g,h)) z_h \bigg\| :  \| (w_k)\| \leq 1, \| (z_k)\| \leq 1  \bigg\},
\intertext{and applying \cite[Corollary~19.6]{ExelBook}, we can continue:}
    &= \| \iota\left( y \right)  \| \\
    &= \sup \bigg\{ \bigg\| \sum_{g, h}  E_A ( w_g^*a(g,h) z_h)  \bigg\| :  \| (w_k)\| \leq 1, \| (z_k)\| \leq 1  \bigg\}  \\
	&\leq \sup \bigg\{ \bigg\| \sum_{g, h}  w_g^*a(g,h) z_h \bigg\| :  \| (w_k)\| \leq 1, \| (z_k)\| \leq 1  \bigg\}  \\
	&\leq  \sup \{ \| \langle (w_k), j(x) (z_k) \rangle :  \| (w_k)\| \leq 1, \| (z_k)\| \leq 1  \}.
\end{align*}
Another application of~\eqref{eq-operator-norm} shows that this is equal to $\| j(x) \|$. Since $\iota$ and $j$ are injective $C^*$-homomorphisms, and hence isometric, we obtain $\|E^0_{\operatorname{stb}}(x)\| \le \|x\|$. Thus $E^0_{\operatorname{stb}}$ is norm-decreasing, and hence extends to a contractive linear map $E_{\operatorname{biv}} \colon A \otimes \Kk(\ell^2(\Gamma)) \to (A \otimes \Kk(\ell^2(\Gamma)))_e$. This map clearly pointwise fixes $(A \otimes \Kk(\ell^2(\Gamma)))_e$ and vanishes on $(A \otimes \Kk(\ell^2(\Gamma)))_g$ for $g \neq e$. It now follows from \cite{Tomi1957} that $E_{\operatorname{biv}}$ is a conditional expectation.  This completes the proof that $A \otimes \Kk(\ell^2(\Gamma))$ is a topologically graded $C^*$-algebra.
\end{proof}

We write $A \otimes \Kk(\ell^2(\Gamma))^{\operatorname{uni}}$ to indicate $A \otimes \mathcal{K}(\ell^2(\Gamma))$ with the univalent grading of Definition~\ref{dfn-univalent-grading}, and $A \otimes \Kk(\ell^2(\Gamma))^{\operatorname{biv}}$ for $A \otimes \mathcal{K}(\ell^2(\Gamma))$  with the bivalent grading of Lemma~\ref{lem-gradstable}.

\begin{thm}\label{thm-BGR-gradedME}
Let $\Gamma$ be a countable, discrete group, let $A$ be a unital topologically $\Gamma$-graded $C^*$-algebra, and let $p$ be a projection in $A_e$ that is full in $A$.  Then $p A p \otimes \Kk(\ell^2(\Gamma))^{\operatorname{biv}} \otimes \Kk(\ell^2(\Gamma))^{\operatorname{uni}}$ is graded isomorphic to $A \otimes  \Kk(\ell^2(\Gamma))^{\operatorname{biv}} \otimes \Kk(\ell^2(\Gamma))^{\operatorname{uni}}$.
\end{thm}

\begin{proof}
Set $P = \sum_{ g } p \otimes \mathbf{e}_{g,g}$, which is an element of the multiplier algebra of $A \otimes \Kk( \ell^2(\Gamma))^{\operatorname{biv}}$.  We claim that $P\left(A \otimes \Kk(\ell^2(\Gamma))^{\operatorname{biv}}\right)_eP$ is full in $\left(A \otimes \Kk(\ell^2(\Gamma))^{\operatorname{biv}}\right)_e$.  Let $I$ be the ideal of $\left(A \otimes \Kk(\ell^2(\Gamma))^{\operatorname{biv}}\right)_e$ generated by $P\left(A \otimes \Kk(\ell^2(\Gamma))^{\operatorname{biv}}\right)_eP$.  Since $\sum_{ g } 1_{A} \otimes \mathbf{e}_{g,g}$ is the identity of the multiplier algebra of $\left(A \otimes \Kk(\ell^2(\Gamma))^{\operatorname{biv}}\right)_e$, to prove that $I = \left(A \otimes \Kk(\ell^2(\Gamma))^{\operatorname{biv}}\right)_e$, it is enough to show that for all $h \in \Gamma$, the element $1_A \otimes \mathbf{e}_{hh}$ belongs to $I$.

Fix $h \in \Gamma$.  Since $p$ is full in the unital $C^*$-algebra $A$, by \cite[Lemma~3.3.6]{HL-book}, there exist $x_1, \ldots, x_n \in A$ such that $1_A = \sum_{ i = 1}^n x_i^* p x_i$.  For each $i$, choose $y_i = \sum_{ g } a_{g,i} \in \operatorname{span} \bigcup_{g} A_g$ such that $\| x_i - y_i \| < \frac{1}{ n(2\max_i\{ \|x_i\| \} + 1)}$.  Then
\begin{align*}
\bigg\| 1_A {}&- E_{A}\bigg(\sum_{ i=1}^n y_i^* p y_i \bigg) \bigg\| \\
    &= \bigg\| E_A\bigg(1_A - \sum_{ i=1}^n y_i^* p y_i \bigg) \bigg\| \\
	&\leq \bigg\| \sum_{ i=1}^n x_i^* p x_i - \sum_{ i=1}^n y_i^* p x_i\bigg\| + \bigg\| \sum_{ i=1}^n y_i^* p x_i - \sum_{ i=1}^n y_i^* p y_i\bigg\| \\
	&\leq \bigg(\sum^n_{i=1} \|(x_i - y_i)^*\| \|x_i\|\bigg) + \bigg(\sum^n_{i=1} \|y_i^*\| \|x_i - y_i\|\bigg)\\
    &= \sum_{i=1}^n\| x_i - y_i \| (\|x_i\| + \|y_i\|) \\
	&<  n\frac{1}{ n(2\max_i\{ \|x_i\| \} + 1)}( 2\max_i\{ \|x_i\| \} +1 )
	  < 1.
\end{align*}
Thus, $E_{A}\left(\sum_{ i=1}^n y_i^* p y_i \right)$ is an invertible positive element in $A_e$.  Hence, $z := \big(E_{A}\big(\sum_{ i=1}^n y_i^* p y_i \big)\big)^{-1} \in A_e$ is positive.  Thus,
\[
    1_A = E_{A}\bigg( \sum_{ i=1}^n z^{1/2} y_i^* p y_i z^{1/2} \bigg)
\]
and
\[
    E_A (z^{1/2} y_i^* p y_i z^{1/2}) = \sum_{g} z^{1/2} a_{g,i}^* p a_{g,i}z^{1/2}
\]
for all $i$.  Thus,
\[
    1_A = \sum_{ i = 1}^n \sum_{ g } z^{1/2} a_{g,i}^* p a_{g,i}z^{1/2}.
\]
Set $b_{g,i} = a_{g,i} z^{1/2}$.  Then
\begin{align*}
1_A \otimes \mathbf{e}_{h,h}
    &= \sum_{ i = 1}^n \sum_{ g } b_{g,i}^* p b_{g,i} \otimes \mathbf{e}_{h,h} \\
	&= \sum_{ i = 1}^n \bigg( \bigg(\sum_{g} b_{g,i}^* \otimes \mathbf{e}_{h , gh} \bigg) \bigg( \sum_{ g } p \otimes \mathbf{e}_{gh, gh} \bigg) \bigg( \sum_g b_{g,i} \otimes \mathbf{e}_{gh, h}\bigg)\bigg) \\
	&= \sum_{ i=1}^n  \bigg(\sum_{g} b_{g,i}^* \otimes \mathbf{e}_{h , gh} \bigg) P \bigg( \sum_g b_{g,i} \otimes \mathbf{e}_{gh, h}\bigg).
\end{align*}
For each $i$, both $b_{g,i} \otimes \mathbf{e}_{gh, h}$ and $b_{g,i}^* \otimes \mathbf{e}_{h, gh}$ belong to $\left(A \otimes \Kk(\ell^2(\Gamma))^{\operatorname{biv}}\right)_e$ because $b_{g,i} \in A_g = A_{ghh^{-1}}$ and $b_{g,i}^* \in A_{g^{-1}} = A_{h( gh)^{-1}}$.  Thus, $\sum_g b_{g,i} \otimes \mathbf{e}_{gh, h}$ and $\sum_{g} b_{g,i}^* \otimes \mathbf{e}_{h ,gh}$ belong to $(A \otimes \Kk(\ell^2(\Gamma)))_e$ for all $i$.  Consequently, $I = \left(A \otimes \Kk(\ell^2(\Gamma))^{\operatorname{biv}}\right)_e$, which is to say that $P\left(A \otimes \Kk(\ell^2(\Gamma))^{\operatorname{biv}}\right)_eP$ is full in $\left(A \otimes \Kk(\ell^2(\Gamma))^{\operatorname{biv}}\right)_e$.

Let $X = P\left(A \otimes \Kk(\ell^2(\Gamma))^{\operatorname{biv}}\right)$.  The previous paragraph implies that $X$ is a graded $P\left(A \otimes \Kk(\ell^2(\Gamma))^{\operatorname{biv}}\right)P$--$A \otimes \Kk(\ell^2(\Gamma))^{\operatorname{biv}}$-bimodule that makes $P\left(A \otimes \Kk(\ell^2(\Gamma))^{\operatorname{biv}}\right)P$ and $A \otimes \Kk(\ell^2(\Gamma))^{\operatorname{biv}}$ homogeneously graded Morita equivalent.  By Lemma~\ref{thm-BGR},
\[
P\left(A \otimes \Kk(\ell^2(\Gamma))^{\operatorname{biv}}\right)P \otimes \Kk(\ell^2(\Gamma))^{\operatorname{uni}} \cong_{\operatorname{gr}} A \otimes \Kk(\ell^2(\Gamma))^{\operatorname{biv}} \otimes \Kk(\ell^2(\Gamma))^{\operatorname{uni}}.
\]
As $P\left(A \otimes \Kk(\ell^2(\Gamma))^{\operatorname{biv}}\right)P \otimes \Kk(\ell^2(\Gamma))^{\operatorname{uni}}$ and $p A p \otimes \Kk(\ell^2(\Gamma))^{\operatorname{biv}} \otimes \Kk(\ell^2(\Gamma))^{\operatorname{uni}}$ are graded isomorphic, the result follows.
\end{proof}

The next result is our second Brown--Green--Rieffel Stabilisation for topologically graded $C^*$-algebras.  We restrict to unital $C^*$-algebras as this covers all of our applications. We expect that the corollary also holds for topologically graded $C^*$-algebras with an approximate unit consisting of projections in the zero component.

\begin{cor}\label{cor-BGR-gradedME}
Let $\Gamma$ be a countable, discrete group, and let $A$ and $B$ be unital topologically $\Gamma$-graded $C^*$-algebras.  Then $A$ and $B$ are graded Morita equivalent if and only if $A  \otimes \Kk(\ell^2(\Gamma))^{\operatorname{biv}} \otimes \Kk(\ell^2(\Gamma))^{\operatorname{uni}}$ is graded isomorphic to $B \otimes  \Kk(\ell^2(\Gamma))^{\operatorname{biv}} \otimes \Kk(\ell^2(\Gamma))^{\operatorname{uni}}$.
\end{cor}

\begin{proof}
(${\Longrightarrow}$) Suppose that $X$ is a graded $A$--$B$-imprimitivity bimodule.  Then the associated linking algebra $L(X)$ is a topologically $\Gamma$-graded $C^*$-algebra. Let $p_A, p_B \in L(X)$ be the elements
\[
p_A = \left(\begin{matrix} 1_A & 0\\ 0 & 0\end{matrix}\right)
    \quad\text{ and }\quad
p_B = \left(\begin{matrix} 0 & 0\\ 0 & 1_B\end{matrix}\right).
\]
Then $p_A, p_B$ belong to $L(X)_e$ and are full projections in $L(X)$, and $A$ is graded isomorphic to $p_A L(X) p_A$ and $B$ is graded isomorphic to $p_B L(X) p_B$.  By Theorem~\ref{thm-BGR-gradedME},
\[
A \otimes \Kk(\ell^2(\Gamma))^{\operatorname{biv}} \otimes \Kk(\ell^2(\Gamma))^{\operatorname{uni}} \cong p_A L(X)p_A \otimes \Kk(\ell^2(\Gamma))^{\operatorname{biv}} \otimes \Kk(\ell^2(\Gamma))^{\operatorname{uni}}
\]
and
\[
B\otimes \Kk(\ell^2(\Gamma))^{\operatorname{biv}} \otimes \Kk(\ell^2(\Gamma))^{\operatorname{uni}} \cong p_B L(X)p_B \otimes \Kk(\ell^2(\Gamma))^{\operatorname{biv}} \otimes \Kk(\ell^2(\Gamma))^{\operatorname{uni}}.
\]
Thus,
\[
A \otimes \Kk(\ell^2(\Gamma))^{\operatorname{biv}} \otimes \Kk(\ell^2(\Gamma))^{\operatorname{uni}} \cong B\otimes \Kk(\ell^2(\Gamma))^{\operatorname{biv}} \otimes \Kk(\ell^2(\Gamma))^{\operatorname{uni}}.
\]

(${\Longleftarrow}$) Suppose that
\[
\psi : A  \otimes \Kk(\ell^2(\Gamma))^{\operatorname{biv}} \otimes \Kk(\ell^2(\Gamma))^{\operatorname{uni}} \to B \otimes  \Kk(\ell^2(\Gamma))^{\operatorname{biv}} \otimes \Kk(\ell^2(\Gamma))^{\operatorname{uni}}
\]
is a graded isomorphism. Then $\psi$ makes
\[
X = (1_A \otimes \mathbf{e}_{e,e} \otimes \mathbf{e}_{e,e})\left( A  \otimes \Kk(\ell^2(\Gamma))^{\operatorname{biv}} \otimes \Kk(\ell^2(\Gamma))^{\operatorname{uni}}\right) \psi^{-1} (1_B \otimes \mathbf{e}_{e,e} \otimes \mathbf{e}_{e,e})
\]
a graded $A$--$B$-imprimitivity bimodule.
\end{proof}

We now discuss a connection with coactions so that we can apply our results to gauge actions (see Corollary~\ref{010}). Given a discrete group $\Gamma$, we write $u_g$ for the image of $g \in \Gamma$ in $C^*(\Gamma)$. We write $\delta_\Gamma:\cstg\to \cstg\xt\cstg$ (we use the minimal tensor product) for the comultiplication such that $\delta_\Gamma(u_g) = u_g \otimes u_g$ for $g \in \Gamma$. Given a $C^*$-algebra $A$, a \emph{coaction} of $\Gamma$ on $A$ is a nondegenerate homomorphism $\delta:A\to A\xt\cstg$ satisfying $\clsp{\delta(A)(1\xt\cstg)}=A\xt\cstg$ and the coaction identity $(\delta \otimes 1) \circ \delta = (1 \otimes \delta_\Gamma) \circ \delta$. The spectral subspaces of $A$ are the spaces $A_g := \{a \in A : \delta(a) = a \otimes u_g\}$.

By \cite[Lemma~1.3]{JQ-JAust1996}, for each $g \in \Gamma$ there is a norm-decreasing linear map $\Phi_g : A \to A_g$ that
fixes $A_g$ pointwise and annihilates $A_h$ for $h \not=g$. Specifically, writing
$\operatorname{Tr}$ for the canonical trace on $C^*_r(\Gamma)$, the map $\Phi_g$ is given
by $\Phi_g(a) = (\operatorname{id}_A \otimes \operatorname{Tr})(\delta(a)(1_A \otimes
\lambda_{g^{-1}}))$ for $a \in A$. In particular $\Phi_e = E_A : A \to A_e$
is a conditional expectation.  By \cite[Lemma~1.3 and Lemma~1.5]{JQ-JAust1996}, the spectral subspaces $\{ A_g\}_{g \in \Gamma}$ and the conditional expectation $E_A$ make $A$ a topologically $\Gamma$-graded $C^*$-algebra.  We refer to this grading on $A$ as the \emph{grading induced by $\delta$}. The connections among gradings, Fell bundles, and coactions are deep (see \cite{JQ-JAust1996}, \cite{EKQR-MemoirsAMS2006}). However, there is a subtlety: while it is true that every coaction $(A,\delta)$ gives rise to a \topo\ grading, this is not reversible: there exist \topo\ gradings that do not arise from a coaction (see \cite[Remark~2.2]{echqui}). There is no problem if the coaction is normal---this corresponds to the associated \CE\ being faithful, and in particular there is no problem if $\Gamma$ is amenable.

We recall the theory of Morita equivalent coactions from \cite{EKQR-MemoirsAMS2006}.
Given coactions $(A,\delta)$ and $(B,\epsilon)$,
and an $A$--$B$-\ibm\ $X$,
a \emph{$\delta$--$\epsilon$ compatible coaction}
on $X$ is a
linear map $\zeta:X\to X\xt\cstg$ such that
for all $x,y\in X$, $a\in A$, and $b\in B$ we have
\begin{itemize}
    \item $\zeta(axb)=\delta(a)\zeta(x)\epsilon(b)$
    \item $\<\zeta(x),\zeta(y)\>_{B\xt \cstg}=\delta(x)^*\delta(y)$
    \item $(\zeta\xt\id)\circ\zeta=(\id\xt\deltag)\circ\zeta$.
\end{itemize}
Note that $\delta$ does map into $X\xt \cstg$---we don't need to work with multiplier bimodules---because $\Gamma$ is discrete.
As explained in \cite[Remark~2.11]{EKQR-MemoirsAMS2006},
a $\delta$--$\epsilon$ compatible coaction as above
automatically satisfies the following,
because $X$ is an \ibm:
\begin{itemize}
    \item $\pre{A\xt\cstg}\<\zeta(x),\zeta(y)\>=\delta(x)\delta(y)^*$
    \item $\clsp\{\zeta(X)(1\xt\cstg)\}=\clsp\{(1\xt\cstg)\zeta(X)\}=X\xt\cstg$.
\end{itemize}

Given a $\delta$--$\epsilon$ compatible coaction $\zeta$ on an $A$--$B$-\ibm\ $X$, let $L(X)=\smtx{A&X\\\wilde X&B}$ be the linking algebra. Then by \cite[Lemma~2.22]{EKQR-MemoirsAMS2006}, the map $\nu = \smtx{\delta&\zeta\\\wilde\zeta&\epsilon}$
is a coaction on $L(X)$ that is normal iff $\epsilon$ is.\footnote{We have used \cite[Section~1.5]{EKQR-MemoirsAMS2006} to identify $L(X)\xt \Kk$ with $L(X\xt \Kk)$, and we similarly identify $L(X)\xt \cstg$ with $L(X\xt \cstg)$.}

Our next goal is to relate Morita equivalence of $C^*$-algebras with coactions and our notion of graded Morita equivalence.  We show that they are equivalent when the group is a discrete amenable group.

\begin{thm}\label{thm-grME-coaction}
Let $\Gamma$ be a discrete, amenable group, let $A_1$ and $A_2$ be $C^*$-algebras, let $\delta_1$ be coaction of $\Gamma$ on $A_1$, and let $\delta_2$ be coaction of $\Gamma$ on $A_2$.
Then $(A_1, \delta_1)$ and $(A_2, \delta_2)$ are Morita equivalent if and only if $A_1$ and $A_2$ are graded Morita equivalent as topologically graded $C^*$-algebras under the gradings induced by $\delta_1$ and $\delta_2$.
\end{thm}

\begin{proof}
Suppose $(A_1, \delta_1)$ and $(A_2, \delta_2)$ are Morita equivalent.  Then there exist an $A_1$--$A_2$-imprimivity bimodule $X$ and
a $\delta_1-\delta_2$ compatible coaction $\zeta$ on $X$.  By \cite[Lemma~2.22]{EKQR-MemoirsAMS2006} (as summarized above),
there exists a unique coaction $\nu$ on the linking algebra $L(X)$ of $X$ given by
\[
\nu\left(\begin{pmatrix} a & x \\ \widetilde{y} & b \end{pmatrix}\right) = \begin{pmatrix} \delta_1(a) & \zeta(x) \\ \widetilde{\zeta(y)} & \delta_2(b) \end{pmatrix}
\]
for all $a \in A_1$, for all $x,y \in X$, and for all $b \in A_2$.

Let $p_1, p_2 \in L(X)$ be the projections
\[
p_1 = \left(\begin{matrix} 1_{\mathcal{M}(A_1)} & 0\\ 0 & 0\end{matrix}\right)
    \quad\text{ and }\quad
p_2 = \left(\begin{matrix} 0 & 0\\ 0 & 1_{\mathcal{M}(A_2)}\end{matrix}\right).
\]
Under the gradings induced by the coactions $\delta_1$, $\delta_2$, and $\nu$, each $A_i$ is graded isomorphic to $p_i L(X) p_i$.  These isomorphisms make $Y =  p_1 L(X) p_2$ an $A_1$--$A_2$-imprimitivity bimodule, which inherits a grading from $L(X)$:
\[
Y_g =  Y \cap L(X)_g
\]
for all $g \in \Gamma$.  Thus, $Y$ is a graded $A_1$--$A_2$-imprimitivity bimodule which implies that $A_1$ and $A_2$ are graded Morita equivalent.

For the converse, assume that there exists a graded $A_1$--$A_2$-imprimitivity bimodule $X$.  By Lemma~\ref{lem-linkalg-grd}, the linking algebra $L(X)$ is a topologically graded $C^*$-algebra.  Let $\mathcal{B}$ be the associated Fell bundle over $\Gamma$ with fibers $L(X)_g$, and let $\Lambda \colon C^*(\mathcal{B}) \to C^*_r(\mathcal{B})$ be the left regular representation of $C^*(\mathcal{B})$ in $C_r^*(\mathcal{B})$. Then by \cite[Theorem~3.3]{Exel-Crelle1997} (see also \cite[Theorem~19.5]{ExelBook}), there are graded surjective $*$-homomorphisms $\lambda \colon L(X) \to C^*_r(\mathcal{B})$ and $\Phi \colon C^*(\mathcal{B}) \to L(X)$ such that $\Lambda = \lambda \circ \Phi$.  Since $\Gamma$ is a discrete amenable group, \cite[Theorem~4.7]{Exel-Crelle1997} implies that $\Lambda$ is an isomorphism. Hence $\lambda$ and $\Phi$ are isomorphisms.  Let $\pi$ be the universal representation of $\mathcal{B}$ in $C^*(\mathcal{B})$. Raeburn proved in \cite[Appendix~A]{IR-DefFell2016} that there is a coaction $\zeta$ of $\Gamma$ on $C^*(\mathcal{B})$ such that $\zeta(b) = \pi_g(b)\otimes u_g$ for $b \in L(X)_g$. Hence $(\Phi \otimes \operatorname{id}_{C^*_r(\Gamma)}) \circ \zeta \circ \Phi^{-1}$ is a coaction of $\Gamma$ on $L(X)$, which, by slight abuse of notation, we denote again by $\zeta$.

Write $p_1, p_2 \in \Mm(L(X))$ for the projections
\[
p_1 = \begin{pmatrix} 1_{\mathcal{M}(A_1)} & 0 \\ 0 & 0 \end{pmatrix}
    \quad\text{ and }\quad
p_2 = \begin{pmatrix} 0 & 0 \\ 0 & 1_{\mathcal{M}(A_2)} \end{pmatrix}.
\]
For $i = 1,2$, the embedding $A_i \hookrightarrow L(X)$ is an injective $\delta_i$--$\zeta$-equivariant homomorphism with range $p_i L(X) p_i$. These embeddings make $Y =  p_1 L(X) p_2$ into an $A_1$--$A_2$-imprimitivity bimodule, and the restriction of $\zeta$ to $Y$ makes it into a Morita equivalence between $(A_1, \delta_1)$ and $(A_2, \delta_2)$.
\end{proof}

Given a finite graph $E$, by \cite[Theorem~3.15]{RH-IJM2013} and \cite[Theorem~3]{CD-JAC2023}, $C^*(E)$ is a \sat graded $C^*$-algebra if and only if $E$ has no sinks. Hence, Corollary~\ref{cor-homogeneous for free} and Theorem~\ref{thm-grME-coaction} give the following.

\begin{cor}\label{010}
Let $E$ and $F$ be finite regular graphs.  The following are equivalent:
\begin{enumerate}
    \item $( C^*(E), \gamma^E)$ and $(C^*(F), \gamma^F)$ are Morita equivalent;
    \item $(C^*(E) \otimes \Kk, \gamma^E \otimes \operatorname{id}) \cong (C^*(F) \otimes \Kk, \gamma^F \otimes \operatorname{id})$; and
    \item For any actions $\rho_1, \rho_2$ of $\TT$ on $\Kk$ that pointwise fix the standard diagonal in $\Kk$,
    \[
        (C^*(E) \otimes \Kk \otimes \Kk, \gamma^E \otimes \rho_1 \otimes \operatorname{id}) \cong (C^*(F) \otimes \Kk \otimes \Kk, \gamma^F \otimes \rho_2 \otimes \operatorname{id}).
    \]
\end{enumerate}
\end{cor}

Let $A$ be an $n \times n$ matrix with entries in $\NN$. Let
\begin{align*}
\Delta_A &\coloneqq \left\{ v \in  \bigcap_{k=1}^\infty  \QQ^n A^k \ | \  v A^\ell \in \ZZ^n, \text{ for some } \ell \in \ZZ^{\geq 0} \right\},\quad\text{ and}\\
\Delta_A^+ &\coloneqq \left\{ v \in \bigcap_{k=1}^\infty  \QQ^n A^k   \ | \ v A^\ell \in \NN^n, \text{ for some } \ell \in \ZZ^{\geq 0} \right\},
\end{align*}
and let $\delta_A \colon \Delta_A \to \Delta_A$ be the automorphism $\delta_A( v) = vA$. Then $(\Delta_A, \Delta_A^+ , \delta_A)$ is called the Krieger Dimension triple of $A$.

\begin{thm}[{\cite[Corollary~4.3]{BK00}}]
Let $A$ and $B$ be primitive matrices with entries in $\NN_0$.  Let $\widetilde{\lambda}$ denote the regular representation of $\TT$ on $L^2(\TT)$. Then the following are equivalent.
\begin{enumerate}
    \item $( \Delta_A , \Delta_A^+, \delta_A) \cong ( \Delta_B , \Delta_B^+, \delta_B)$;
    \item $( \mathcal{O}_A \otimes \Kk( L^2(\mathbb{T}) ) , \gamma_A \otimes \operatorname{id}) \cong  ( \mathcal{O}_B \otimes \Kk( L^2(\mathbb{T}) ) , \gamma_B \otimes \operatorname{id})$;
    %\jq{$B$, $\gamma_B$?}
    \item $(\mathcal{O}_A \otimes \Kk( L^2(\mathbb{T}) ) , \gamma_A \otimes \operatorname{Ad} \widetilde{\lambda}) \cong (\mathcal{O}_B \otimes \Kk( L^2(\mathbb{T}) ) , \gamma_B \otimes \operatorname{Ad} \widetilde{\lambda})$;
    \item $(\mathcal{O}_A, \gamma_A)$ and $(\mathcal{O}_B, \gamma_B)$ are Morita equivalent; and
    \item $(\mathcal{O}_A \otimes \Kk \otimes \Kk, \gamma_A \otimes \rho_1 \otimes \operatorname{id}) \cong (\mathcal{O}_B \otimes \Kk \otimes \Kk , \gamma_B \otimes \rho_2\otimes \operatorname{id})$, for any actions $\rho_1, \rho_2$ of $\TT$ on $\Kk$ that pointwise fix the standard diagonal in $\Kk$.
\end{enumerate}
\end{thm}

\begin{proof}
First note that $\mathcal{O}_A$ and $\mathcal{O}_B$ are \sat graded $C^*$-algebras.  Corollary~\ref{cor-homogeneous for free} and \cite[Corollary~4.3]{BK00} imply that~(1), (2), (3)~and~(4) are equivalent.  It is clear that~(2) implies~(5).  For (5)~implies~(4), let $\varphi \colon \mathcal{O}_A \otimes \Kk \otimes \Kk \to \mathcal{O}_B \otimes \Kk \otimes \Kk$ be an isomorphism that intertwines $\gamma_A \otimes \rho_1 \otimes \operatorname{id}$ and $\gamma_B \otimes \rho_1 \otimes \operatorname{id}$. Then $\varphi(1_{\mathcal{O}_A \otimes \mathbf{e}_{1,1} \otimes \mathbf{e}_{1,1} }) \left( \mathcal{O}_B \otimes \Kk \otimes \Kk\right)(1_{\mathcal{O}_B \otimes \mathbf{e}_{1,1} \otimes \mathbf{e}_{1,1} })$ is an  $\mathcal{O}_A$--$\mathcal{O}_B$-imprimitivity bimodule.
\end{proof}

\section{Incorporating Cartan subalgebras}

In this section we describe how the results of the preceding section can be modified to incorporate the presence of a Cartan subalagebra with totally disconnected spectrum in each of the $C^*$-algebras $A$ and $B$. Specifically, we show that given a graded Morita equivalence from $(A_1, D_1)$ to $(A_2, D_2)$, there is a diagonal-preserving stable isomorphism that intertwines the univalent gradings.

Our strategy is first to use the argument of the preceding section to obtain a diagonal-preserving Morita equivalence between the $0$-graded subalgebras of $A$ and $B$. We then appeal to Matsumoto's work to find an isometry that implements a stable isomorphism of Cartan pairs at the level of 0-graded components. This same isometry implements a stable isomorphism of the original Cartan pairs; and it is a graded isomorphism (with respect to the univalent grading) because the implementing isometry is 0-graded.

\begin{dfn}\label{dfn-moritaeq-diag}
Let $A_1, A_2$ be $C^*$-algebras and let $D_i$ be a subalgebra of $A_i$ for $i=1,2$.  We say that $(A_1, D_1)$ and $(A_2, D_2)$ are \emph{Morita equivalent} provided that there is an $A_1$--$A_2$-imprimitivity bimodule such that
\[
X = \clsp \{ x \in X : \langle x , D_1 \cdot x \rangle_{A_2} \subseteq D_2 \quad \text{and} \quad _{A_1}\langle x \cdot D_2, x \rangle \subseteq D_1 \}
\]
We call $X$ an \emph{$(A_1, D_1)$--$(A_2, D_2)$-imprimitivity bimodule}.
\end{dfn}

The following technical lemma is the key to our approach.

\begin{lem}\label{lem:nice frame CStar version}
Let $A_1, A_2$ be separable $C^*$-algebras, and suppose that $D_i \subseteq A_i$ are $\sigma$-unital abelian subalgebras with totally disconnected spectra. Suppose that $X$ is an $(A_1, D_1)$--$(A_2, D_2)$-imprimitivity bimodule, so that
\[
N_X := \{x \in X : {_{A_1}\langle x \cdot D_2, x\rangle}\subseteq D_1\text{ and }
                    {\langle x, D_1 \cdot x\rangle_{A_2}}\subseteq D_2\}
\]
densely spans $X$. Suppose that $\clsp\{\langle x, x\rangle_{A_2} : x \in N_X\} = D_2$. Then there is a sequence $(x_i)_i$ in $X$ such that
\begin{enumerate}
    \item $\langle x_i, D_1 \cdot x_i\rangle_{A_1} \in D_2$ for all $i$,
    \item $\sum^\infty_{i=1} \langle x_i, x_i\rangle_{A_2} = 1_{\Mm(D_2)}$ in the strict topology, and
    \item $\langle x_i, x_j\rangle_{A_2} = 0$ for $i \not= j$. %\jq{$\<x_i,x_j\>$}
\end{enumerate}
\end{lem}
\begin{proof}
Throughout this proof, we freely identify $D_2$ with $C_0(\widehat{D}_2)$ via the Gelfand transform.  Fix $\phi \in \widehat{D}_2$. Since $\clsp\{\langle x, x\rangle_{A_2} : x \in N_X\} = D_2$, there exists $y_\phi \in N_X$ such that $\langle y_\phi, y_\phi\rangle_{A_2}(\phi) = 1$. Since $\widehat{D}_2$ is totally disconnected, there is a compact open neighbourhood $U_\phi$ of $\phi$ such that $\langle y_\phi, y_\phi\rangle_{A_2}(\psi) \ge \frac{1}{2}$ for all $\psi \in U_\phi$. Define $d_\phi \in D_2$ by $d(\psi) = \mathbbm{1}_{U_\phi}(\psi) \langle y_\phi, y_\phi\rangle_{A_2}(\psi)^{-1/2}$. Then $\langle y_\phi \cdot d_\phi, y_\phi \cdot d_\phi\rangle_{D_2} = \mathbbm{1}_{U_\phi}$.

The sets $U_\phi, \phi \in \widehat{D}_2$ obtained from the preceding paragraph form an open cover of $\widehat{D}_2$ by compact open sets. Since $D_2$ is separable, $\widehat{D}_2$ is second countable, and so there is a sequence $(\phi_i)_i$ in $\widehat{D}_2$ such that $(U_{\phi_i})_{i=1}^\infty$ covers $\widehat{D}_2$. Define a sequence $p_i \in D_2$ inductively by $p_1 = \mathbbm{1}_{U_{\phi_1}}$, and $p_{i+1} = \mathbbm{1}_{U_{\phi_{i+1}}} - \sum^i_{j=1} p_j$. Then the $p_i$ are mutually orthogonal, and $p_i \le \langle y_{\phi_i} \cdot d_{\phi_i}, y_{\phi_i} \cdot d_{\phi_i}\rangle_{D_2}$ for all $i$. Put $x_i = y_{\phi_i} \cdot d_{\phi_i} p_i$ for each $i$. Then $\langle x_i, x_i\rangle_{D_2} = p_i$ for all $i$.

We check that the $x_i$ have the desired properties. For each $i$ we have
\begin{align*}
\langle x_i, D_1 \cdot x_i\rangle_{A_1}
    &= \langle y_{\phi_i} \cdot d_{\phi_i} p_i, D_1 y_{\phi_i} \cdot d_{\phi_i} p_i\rangle_{A_1}\\
    &= (d_{\phi_i} p_i)^* \langle y_{\phi_i}, D_1 y_{\phi_i}\rangle_{D_2} d_{\phi_i} p_i
    \subseteq (d_{\phi_i} p_i)^* D_2 d_{\phi_i} p_i \subseteq D_2,
\end{align*}
which gives~(1). Since $\bigcup_i \supp p_i = \bigcup_i U_i = \widehat{D}_2$ and the $p_i$ are mutually orthogonal, the series $\sum_i p_i$ converges strictly to $1_{\Mm(D_2)}$, giving~(2). For~(3), fix $i \not= j$. Since $p_i$ is a projection, we have $x_i = y_{\phi_i} \cdot d_{\phi_i} p_i = y_{\phi_i} \cdot d_{\phi_i} p_i^2 = x_i \cdot p_i$, and similarly for $x_j$, and so $\langle x_i, x_j\rangle_{D_2} = \langle x_i\cdot p_i, x_j \cdot p_j\rangle_{D_2} = p_i \langle x_i, x_j\rangle_{D_2} p_j = p_i p_j \langle x_i, x_j\rangle_{D_2}$ because $D_2$ is abelian. Since the $p_i$ are mutually orthogonal, we deduce that $\langle x_i, x_j\rangle_{D_2} = 0$, giving~(3).
\end{proof}

For an abelian subalgebra $D$ of a graded $C^*$-algebra $A$, we write $N_g(D)$ for the set of all homogeneous normalizers of $D$ of degree $g$; that is, $N_g(D)$ is the set of elements $u \in A_g$ satisfying
\[
u D u^* \cup u^* D u \subseteq D.
\]
We write $N_*(D)$ for the union $\bigcup_g N_g(D)$.

We denote the abelian subalgebra of $\Kk$ consisting of diagonal operators in $\Kk$ by $\mathcal{C}$.

\begin{thm}\label{thm-cartan-imprimitivity}
  Let $A_1$ and $A_2$ be separable graded $C^*$-algebras and let $D_i$ be an abelian subalgebra of $A_i$ that contains an approximate identity for $A_i$ and satisfies $A_i = \overline{\operatorname{span}}N_*(D_i)$.
  Suppose that $D_1$ and $D_2$ have totally disconnected spectra, and that $A_1$ and $A_2$ are topologically graded.  Then the followiing are equivalent:
  \begin{enumerate}
      \item  there exists a graded imprimitivity $A_1$--$A_2$-bimodule $X$ such that
        \begin{equation}\label{eq:N spans X}
            X_g = \clsp\{ x \in X_g :  {}_{A_1} \langle x \cdot D_2, x \rangle \subseteq D_1 \text{ and } \langle x, D_1 \cdot x \rangle_{A_2} \subseteq D_2\}
        \end{equation}
        for all $g \in \Gamma$ and such that $X_e$ is an $(A_1)_e$--$(A_2)_e$-imprimitivity bimodule; and
        \item there exists a graded isomorphism between $A_1 \otimes \Kk^{\operatorname{uni}}$ and $A_2 \otimes \Kk^{\operatorname{uni}}$ that preserves the abelian subalgebras $D_1 \otimes \mathcal{C}$ and $D_2 \otimes \mathcal{C}$. %\jq{Remove mathcal? What is $\mathcal C$?}.
    \end{enumerate}
\end{thm}
\begin{proof}
\mbox{(2)\;$\implies$\;(1).} Suppose that there is a graded diagonal-preserving isomorphism $\varphi \colon A_1 \otimes \Kk \to A_2 \otimes \Kk$ that is graded with respect to the univariant gradings. Let $\{ \mathbf{e}_{i,j}\}$ be the standard system of matrix units for $\Kk$ (since we are using univalent gradings, we can index the matrix units by $\NN$). Let $p \coloneqq \varphi(1_{A_1} \otimes \mathbf{e}_{1,1})$ and $q \coloneqq 1_{A_2} \otimes \mathbf{e}_{1,1}$, and let $X \coloneqq p(A_2 \otimes \Kk)q$ with the standard inner products and actions. Fix $g \in \Gamma$. Since $A_2 = \clsp N_*(D_2)$, we have $(A_2)_g = \clsp N_g(D_2)$, %\jq{Similarly, I don't know what ``$N_g(D_2)$'' is.},
and hence $X_g = \clsp p(N_g(D_2) \otimes \Kk)q$.  A direct calculation shows that if $n \in N_g(D_2)$ and $x = n \otimes \mathbf{e}_{i,j}$, then $x^* D_1 x \subseteq D_2$ and $x D_2 x^* \subseteq D_1$.  As $X_g = \clsp \{ p(n \otimes \mathbf{e}_{i,j})q : n \in N_g(D_2) \ \text{and}\ i, j \in \NN \}$, we have that $X$ is a graded imprimitivity bimodule satisfying~\eqref{eq:N spans X}.  Moreover, $X_e = p ((A_2)_e \otimes \Kk)q$ is an $(A_1)_e$--$(A_2)_e$-imprimitivity bimodule since $\varphi$ induces an isomorphism from $(A_1)_e \otimes \Kk$ to $(A_2)_e \otimes \Kk$.

\mbox{(1)\;$\implies$\;(2).} Suppose that $X$ is a graded $A_1$--$A_2$-imprimitivity bimodule satisfying~\eqref{eq:N spans X} and $X_e$ is an $(A_1)_e$--$(A_2)_e$-imprimitivity bimodule.  Then $D_2 \subseteq (A_2)_e = \langle X_e, X_e\rangle_{A_2}$. Let
\[
    N_{X_e} \coloneqq \{ x \in X_e : \langle  x , D_1  \cdot x \rangle_{A_2} \subseteq D_2 \quad \text{ and } _{A_1}\langle  x\cdot D_2 ,  x \rangle \subseteq D_1 \}.
\]
We claim that $\clsp\{\langle x, x\rangle_{A_2} : x \in N_{X_e}\} = D_2$.

To see this, first note that since $D_2 \subseteq (A_2)_e = \langle X_e, X_e\rangle_{A_2}$, Equation~\eqref{eq:N spans X} implies that $\clsp\{\langle x, y\rangle_{A_2} : x \in N_{X_e}\} = D_2$. In particular, identifying $D_2$ with $C_0(\widehat{D}_2)$, the functions $\{\langle x, y\rangle_{D_2} : x,y \in N_{X_e}\}$ vanish nowhere and separate points, and therefore the functions $\langle x, y\rangle_{A_2}^*\langle x, y\rangle_{A_2}$ have the same property. So it suffices to show that each $\langle x, y\rangle_{A_2}^*\langle x, y\rangle_{A_2}$ has the form $\langle z, z\rangle_{A_2}$ for some $z \in N_{X_e}$. To see this, we first calculate, using the imprimitivity condition at the third inequality:
\begin{align*}
\langle x, y\rangle_{A_2}^*\langle x, y\rangle_{A_2}
    &= \langle y, x\rangle_{A_2}\langle x, y\rangle_{A_2} \\
    &= \big\langle y, x\cdot\langle x, y\rangle_{A_2}\big\rangle_{A_2}\\
    &= \big\langle y, {_{A_1}\langle x, x\rangle}\cdot y\rangle_{A_2}\\
    &= \big\langle {_{A_1}\langle x, x\rangle}^{1/2}\cdot y, {_{A_1}\langle x, x\rangle}^{1/2}\cdot y\rangle_{A_2}.
\end{align*}
Since $x \in N_{X_e}$, we have $d \coloneqq {_{A_1}\langle x, x\rangle}^{1/2} \in D_1 \subseteq (A_1)_e$, so $d\cdot y \in N_{X_e}$.

Lemma~\ref{lem:nice frame CStar version} applied to $X_e$ therefore yields a sequence $(x_i)$ in $X_e$ satisfying Condition~2 of \cite[Definition~3.1]{KM-TAMS2018}. The same argument applied to the conjugate module $X^*_e$ shows that $X_e$ contains a sequence $(y_i)$ satisfying Condition~(3) of \cite[Definition~3.1]{KM-TAMS2018}.

Let $A \coloneqq A_1 \oplus X \oplus X^* \oplus A_2$ be the linking algebra of $X$ and let $D = D_1 \oplus D_2 \subseteq A$. Endow $A$ with the grading inherited from those on $A_1$, $A_2$ and $X$. By \cite[Proposition~4.6]{KM-TAMS2018}, there exist isometries $V_1, V_2 \in \Mm(A_e \otimes \Kk)$, regarded as a subset of $\Mm(A \otimes \Kk)$, satisfying the conditions of that proposition. Since the $V_i$ are zero-graded, conjugation by $V_i$ is a diagonal-preserving graded isomorphism from $A_i \otimes \Kk$ to $A \otimes \Kk$. Composing these isomorphisms gives the desired isomorphism of graded stabilisations.
\end{proof}

\begin{cor}
  Let $A_1$ and $A_2$ be separable \sat graded $C^*$-algebras. For each $i$, let $D_i$ be an abelian subalgebra of $A_i$ that contains an approximate identity for $A_i$ and satisfies $A_i = \operatorname{\overline{\text{span}}}N_*(D_i)$.  Suppose that the $D_i$ have totally disconnected spectra. Then there exists a graded imprimitivity $A_1$--$A_2$-bimodule $X$ such that
  \begin{equation*}
  X_g = \clsp\{ x \in X_g :  {}_{A_1} \langle x D_2, x \rangle \subseteq D_1 , \langle x , D_1 x \rangle_{A_2} \subseteq D_2\}
  \end{equation*}
  for all $g \in \Gamma$ if and only if there is a graded isomorphism between $A_1 \otimes \Kk^{\operatorname{uni}}$ and $A_2 \otimes \Kk^{\operatorname{uni}}$ that preserves the Cartan subalgebras $D_1 \otimes \mathcal{C}$ and $D_2 \otimes \mathcal{C}$.
\end{cor}

\begin{proof}
If $X$ is a graded imprimitvity $A_1$--$A_2$-bimodule, then Theorem~\ref{cor-homogeneous for free} implies that $X_e$ is an $(A_1)_e$--$(A_2)_e$-imprimitivity module as $A_1$ and $A_2$ are saturated graded.  The corollary now follows from Theorem~\ref{thm-cartan-imprimitivity}.
\end{proof}

\end{document}